\documentclass[12pt]{amsart}

\usepackage{color,graphicx,amssymb,latexsym,amsfonts,txfonts,amsmath,amsthm}
\usepackage{pdfsync}
\usepackage{amsmath,amscd}
\usepackage[all,cmtip]{xy}

\usepackage{hyperref}
\hypersetup{
    colorlinks=true,       
    linkcolor=blue,          
    citecolor=blue,        
    filecolor=blue,      
    urlcolor=blue           
}

\input epsf

\newtheorem{theo}{Theorem} 
\newtheorem*{conj}{Classical Schottky uniformization conjecture}

\newtheorem{lemm}{Lemma}

\theoremstyle{remark}
\newtheorem{rema}{\bf Remark}

\begin{document}

\title{Towards a proof of the Classical Schottky Uniformization Conjecture} 
\date{\today}

\author{Rub\'en A. Hidalgo} 
\thanks{Partially supported by Project Fondecyt 1230001} 
\keywords{Riemann Surfaces,  Schottky Groups} 
\subjclass[2000]{30F10, 30F40}
\address{Departamento de Matem\'atica y Estad\'{\i}stica, Universidad de La Frontera,  Temuco, Chile}
\email{ruben.hidalgo@ufrontera.cl}

\begin{abstract}
By Koebe's retrosection theorem, every closed Riemann surface of genus $g \geq 2$ is uniformized by a Schottky group. Marden observed that there are Schottky groups that are not classical ones, that is, they cannot be defined by a suitable collection of circles. This opened the question of whether every closed Riemann surface can be uniformized by a classical Schottky group. In this paper, we observe that every Belyi curve can be uniformized by a classical Schottky group. Since Belyi curves form a dense locus 
in the moduli space ${\mathcal M}_{g}$ and the locus ${\mathcal M}_{g}^{cs} \subset {\mathcal M}_{g}$ of those Riemann surfaces uniformized by classical Schottky groups is a non-empty open set, this ensures that ${\mathcal M}_{g}^{cs}$ is open and dense in ${\mathcal M}_{g}$.
\end{abstract}

\maketitle

\section{Introduction}
A {\it Schottky group of rank $g \geq 2$} is defined as (equivalent definitions can be found in \cite{Maskit2}) a group $G$ generated by $g$ loxodromic elements $A_{1},\ldots, A_{g}$, where there exists a collection of
$2g$ pairwise disjoint simple loops $C_{1},\dots,C_{g}$, $C'_{1}, \ldots,C'_{g}$ on the the Riemann sphere $\widehat{\mathbb C}$
bounding a common region $\mathcal D$ of connectivity $2g$ so that $A_{j}(C_{j})=C'_{j}$ and $A_{j}({\mathcal D}) \cap {\mathcal D} = \emptyset$, for all $j=1,\ldots,g$.  The domain $\mathcal D$ is called a {\it standard fundamental domain} for $G$, the collection of loops $C_{1},\ldots,C_{g}$, $C'_{1},\ldots,C'_{g}$ a {\it fundamental set of loops} and the M\"obius transformations $A_{1},\ldots,A_{g}$ a {\it Schottky set of generators}.  
It is known that $G$ is a free group of rank $g$ and, in \cite{Chuckrow}, Chuckrow proved that every set of $g$ generators of it is a Schottky set of generators. Its region of discontinuity $\Omega$ is connected and the quotient space $\Omega/G$ is a closed Riemann surface of genus $g$ \cite{Maskit:Kleinian Groups}.

If $S$ is a closed Riemann surface, then 
Koebe's retrosection theorem states that there is a Schottky group $G$, with region of discontinuity $\Omega$, such that $\Omega/G$ is biholomorphic to $S$ (a simple proof of this fact was also given by L. Bers in \cite{Bers} using quasiconformal mappings). In this case, we say that $S$ is uniformized by $G$.

A Schottky group is called {\it classical} if it has a Schottky set of generators with a fundamental set of loops consisting of circles. 
 Marden \cite{marden:schottky} showed that, for every $g \geq 2$, there are non-classical Schottky groups of rank $g$; see also
\cite{J-M-M:schottky}. An explicit family of examples of non-classical Schottky groups of rank two was
constructed by Yamamoto \cite{yamamoto:example} and a theoretical construction of an infinite collection of non-classical Schottky groups is provided in \cite{H-M}. In \cite{Hou1,Hou2} Hou proved that every convex cocompact Kleinian group of Hausdorff dimension less than one is a classical Schottky group.
This opened the question of whether every closed Riemann surface has a classical Schottky uniformization (I think it was Lipman Bers who first asked this question?).

\begin{conj}
Every closed Riemann surface can be uniformized by a classical Schottky group.
\end{conj}

An affirmative answer to the above conjecture was known in either one of the following two situations.
\begin{enumerate}
\item $S$ admits an anticonformal automorphism of order two with fixed points (Koebe).
\item $S$ has $g$ pairwise disjoint homologically independent short loops (McMullen). 
\end{enumerate}

Let ${\mathcal M}_{g}$ be the moduli space of closed Riemann surfaces of genus $g \geq 2$, ${\mathcal M}_{g}^{cs}$ be its locus consisting of classes of Riemann surfaces which can be uniformized by classical Schottky groups.

\begin{theo}\label{conj}
${\mathcal M}_{g}^{cs}$ is an open and dense subset of ${\mathcal M}_{g}$.
\end{theo}

The main idea in proving the above result is the following. Let ${\mathcal M}_{g}^{b}$ be the locus of classes of Belyi curves. It is well known that ${\mathcal M}_{g}^{cs}$  is a non-empty open set and that ${\mathcal M}_{g}^{b}$ is a dense subset (as a consequence of Belyi's theorem \cite{Belyi}). So, to prove Theorem \ref{conj}, it is enough to check that every Belyi curve can be uniformized by a classical Schottky group (Theorem \ref{belyischottky}).

\begin{rema}
It is important to remark that we are dealing with all Belyi curves and not just the regular Belyi curves (also called quasiplatonic curves); as these last ones only provide a finite collection of points in ${\mathcal M}_{g}$.
\end{rema}

Moduli space ${\mathcal M}_{g}$ is non-compact. A compactification, the Deligne-Mumford compactification, is obtained by considering stable Riemann surfaces of genus $g$. Every stable Riemann surface of genus $g \geq 2$ can be uniformized by a noded Schottky group of rank $g$ (a geometrically finite Kleinian group isomorphic to the free group of rank $g$). These noded Schottky groups are geometric limits of Schottky groups. Similarly to Schottky groups, noded Schottky groups can be defined geometrically using a system of $2g$ simple loops, but one now permits tangencies at parabolic fixed points. A neoclassical Schottky group is one for which this system of loops can be chosen to be circles. In \cite{H-M} it was observed that there are noded Riemann surfaces that cannot be uniformized by a neoclassical Schottky group.

\section{Preliminaries}
\subsection{Annuli and their modules}
An {\it annulus} (or ring domain) is a doubly connected domain $A$ in $\widehat{\mathbb C}$. There exists a unique $r>1$ so that $A$ is biholomorphically equivalent to a circular annulus $A_{r}:=\{z \in {\mathbb C}: r^{-1}<|z|<r\}$. The {\it modulus} of $A$ is defined as ${\rm mod}(A)=\frac{1}{\pi} \log r$. Next, we list some known results on the modulus of annuli.

\begin{lemm}[Gr\"otzch inequality]\label{lema1}
If $A$ is an annulus and $B \subset A$ an essential annulus (i.e., the inclusion map induces an injective map on the fundamental group), then ${\rm mod}(B) \leq {\rm mod}(A)$.
\end{lemm}

\begin{lemm}\label{lema2}
If $A$ and $B$ are annuli and $Q:A \to B$ is a degree $d$ covering map, then ${\rm mod}(A)=d{\rm mod}(B)$.
\end{lemm}

\begin{lemm}\label{lema3}
Every annulus $A$ satisfying that ${\rm mod}(A)>1/2$ contains an Euclidean circle separating its borders.
\end{lemm}

\begin{lemm}\label{lema4}
Let $A$ and $B$ annuli and $Q:A \to B$ be a finite degree surjective holomorphic map (with a finite set of critical points on $A$). We also assume that $Q$ sends a central loop of $A$ onto a central loop of $B$ and this restriction is a covering map of loops. Then there exists a positive constant $C(Q)$, only depending on $Q$, so that
${\rm mod}(A)>C(Q){\rm mod}(B)$.
\end{lemm}

\section{Belyi curves can be uniformized by classical Schottky groups}\label{Sec:Belyi}
A closed Riemann surface $S$, of genus $g \geq 2$, is called a {\it Belyi curve} if it admits a non-constant meromorphic map $\beta:S \to \widehat{\mathbb C}$ whose branch values are contained in the set $\{1,\omega_{3}, \omega_{3}^{2}\}$, where $\omega_{3}=e^{2 \pi i/3}$ (in this case $\beta$ is a called a {\it Belyi map} for $S$). Usually, many authors assume the branch values to be contained in the set $\{\infty,0,1\}$ (as the group of M\"obius transformations acts $3$-transitively on the Riemann sphere).

The group of M\"obius transformations keeping invariant the set $\{1,\omega_{3}, \omega_{3}^{2}\}$ is generated by the transformations $A(z)=\omega_{3} z$ and $B(z)=1/z$ and it is isomorphic to the symmetric group in three letters ${\mathfrak S}_{3}$. The rational map 
$$R(z)=\frac{(1+2\omega_{3})(z^{3}+z^{-3})-6}{(1+2\omega_{3})(z^{3}+z^{-3})+6}$$
provides a regular branched cover with deck group $\langle A,B\rangle$ and whose set of branch values is $\{1,\omega_{3}, \omega_{3}^{2}\}$.

If $\beta:S \to \widehat{\mathbb C}$ is a Belyi map for $S$, then the composition map $R \circ \beta:S \to \widehat{\mathbb C}$ still a Belyi map for $S$. A Belyi map for $S$ obtained as a finite sequence of compositions $R \circ R \circ \cdots \circ R \circ \beta$ is called a {\it refining} of $\beta$. 

\begin{theo}\label{belyischottky}
Every Belyi curve can be uniformized by a classical Schottky group.
\end{theo}
\begin{proof}
Let us denote by $S^{1}=\{a \in {\mathbb C}: |z|=1\}$ the unit circle in the complex plane. If $\beta:S \to \widehat{\mathbb C}$ is a Belyi map for $S$, then $\beta^{-1}(S^{1})$ provides a triangulation of $S$. By taking a refining of $\beta$, if necessary, we may assume the following two properties:
\begin{enumerate}
\item there are $g$ pairwise disjoint homologically independent simple loops $\alpha_{1},\ldots,\alpha_{g}$ inside $\beta^{-1}(S^{1})$;
\item for every $j=1,\ldots,g$, $\beta(\alpha_{j})=S^{1}$ and $\beta:\alpha_{j} \to S^{1}$ is a covering map.
\end{enumerate}

Let us consider a Schottky uniformization $(G,\Omega,P:\Omega \to S)$ defined by the loops $\alpha_{1},\ldots,\alpha_{g}$, that is, there is a fundamental set of loops $C_{1},\dots,C_{g}$, $C'_{1}, \ldots,C'_{g}$ for $G$ so that $P(C_{j})=P(C'_{j})=\alpha_{j}$, for $j=1,\ldots,g$. Let $B_{1},\ldots,B_{g}$ a corresponding Schottky set of generators.

Let us consider the annuli $A_{r}=\{z \in {\mathbb C}: r^{-1} < |z| < r\}$, where $r>1$. Then, the preimage of $A_{r}$ by the meromorphic map $Q=\beta \circ P:\Omega \to \widehat{\mathbb C}$ provides a neighborhood $\widetilde{A}_{r}$ of the graph $\beta^{-1}(S^{1})$. 

Inside $\widetilde{A}_{r}$ there is a collection of $g$ pairwise disjoint annuli, say $\widetilde{A}_{r}^{1},\ldots,\widetilde{A}_{r}^{g}$, where $C_{j} \subset \widetilde{A}_{r}^{j}$. We chose them so that $Q:\widetilde{A}_{r}^{j} \to A_{r}$ is surjective.

In this way, ${\rm mod}(\widetilde{A}_{r}^{j})=C(Q){\rm mod}(A_{r})=C(Q)\log(r)/\pi$ (see Lemma \ref{lema4}).

Next, we make $r$ to approach $+\infty$ to assume the module of each annuli $\widetilde{A}_{r}^{j}$ to be as big as we want. Now Lemma \ref{lema3} ensures that inside $\widetilde{A}_{r}^{j}$ there is a circle $D_{j}$ (homotopic to $C_{j}$ in the corresponding annuli). It can be seen that the new loops $D_{1},\ldots,D_{g}$, $D'_{1}=B_{1}(D_{1}),\ldots,D'_{g}=B_{g}(D_{g})$ define a new set of fundamental loops for $G$ making it a classical Schottky group.
\end{proof}



\begin{thebibliography}{99}


\bibitem{Belyi}
G. V. Belyi.
On Galois extensions of a maximal cyclotomic field.
{\it Mathematics of the USSR-Izvestiya} {\bf 14} No.2 (1980), 247--256.




\bibitem{Bers}
Bers, L. 
Automorphic forms for Schottky groups. 
{\it Adv. in Math.} {\bf 16} (1975), 332-361.



\bibitem{Chuckrow}
Chuckrow, V. 
On Schottky groups with application to Kleinian
groups. {\it Ann. of Math.} {\bf 88} (1968), 47-61.


\bibitem{H-M}
Hidalgo, R.A. and Maskit, B. 
On neoclassical Schottky groups.
{\it Trans. Amer. Math. Soc.} {\bf 358}  (2006),  No.11,  4765-4792. 

\bibitem{Hou1}
Hou, Y. 
All finitely generated Kleinian groups of small Hausdorff dimension are classical Schottky groups.
{\it Math. Z.} {\bf 294} (2020), 901--950.

\bibitem{Hou2}
Hou, Y. 
The classification of Kleinian groups of Hausdorff dimensions at most one.
{\it Q. J. Math.} {\bf 74} (2023), 607--625.





\bibitem{J-M-M:schottky}
J{\o}rgensen, T. Marden, A. and Maskit, B.
The boundary of classical {S}chottky space.
{\it Duke Math. J.} {\bf 46} (1979),441-446.


\bibitem{marden:schottky}
Marden, A.
Schottky groups and circles.
In {\em Contributions to Analysis}, pages 273--278. Academic Press,
New York and London, 1974.


\bibitem{Maskit:Kleinian Groups}
Maskit, B. 
{\it Kleinian Groups}. G.M.W. {\bf 287}, Springer-Verlag, 1988.


\bibitem{Maskit2}
Maskit, B. 
A characterization of Schottky groups. 
{\it J. d'Analyse Math.} {\bf 19} (1967), 227-230.








\bibitem{yamamoto:example}
Yamamoto, Hiro-o.
An example of a non-classical {S}chottky group.
{\em Duke Math. J.}, 63:193--197, 1991.



\end{thebibliography}
\end{document}